\documentclass[a4paper, 12pt]{article}
\usepackage{mathrsfs}
\usepackage{amsmath}
\usepackage{amsfonts}
\usepackage[T1]{fontenc}
\usepackage[latin9]{inputenc}
\usepackage{amsthm}
\usepackage{graphicx}
\usepackage{amsmath}

\usepackage{amsthm}
\usepackage{array}
\usepackage{cases}
\makeatletter
\def\th@plain{%
  \itshape 
}
\makeatother

\makeatletter
\renewenvironment{proof}[1][\proofname]{\par
  \pushQED{\qed}%
  \normalfont \topsep6\p@\@plus6\p@\relax
  \trivlist
  \item[\hskip\labelsep
        \bfseries
    #1\@addpunct{.}]\ignorespaces
}{%
  \popQED\endtrivlist\@endpefalse
}
\makeatother

\numberwithin{equation}{section}

\newtheorem{thm}{Theorem}[section]

\newtheorem{prop}{Proposition}

\newtheorem{clm}{Claim}

\numberwithin{equation}{section}

\usepackage{graphicx, epsfig, subfigure}
\usepackage{enumerate} 
\usepackage[square, numbers, sort&compress]{natbib} 
\ifx\pdfoutput\undefined
 \usepackage[dvipdfm,%
  pdfstartview=FitH, 
  bookmarks=true,%
  bookmarksnumbered=true, 
  bookmarksopen=true, 
  plainpages=false,%
  pdfpagelabels,%
  colorlinks=true, 
  linkcolor=blue, 
  citecolor=blue,%
  urlcolor=black,
  pdfborder=001]{hyperref}
  \AtBeginDvi{}  
\else
 \usepackage[pdftex,%
  pdfstartview=FitH, 
  bookmarks=true,%
  bookmarksnumbered=true, 
  bookmarksopen=true, 
  plainpages=false,%
  pdfpagelabels,%
  colorlinks=true, 
  linkcolor=blue, 
  citecolor=blue,%
  urlcolor=black,
  pdfborder=001]{hyperref}
\fi

\numberwithin{equation}{section}

\begin{document}

\title{\LARGE Equitable tree-$O(d)$-coloring of $d$-degenerate graphs
\thanks{Supported by the National Natural Science Foundation of China (No.\,11871055) and the Youth Talent Support Plan of Xi'an Association for Science and Technology (No.\,2018-6).}
\thanks{Mathematics Subject Classification (2010): 05C15.}
}
\author{Xin Zhang~\thanks{Corresponding author. Emails: xzhang@xidian.edu.cn (X.\,Zhang) beiniu@stu.xidian.edu.cn (B.\,Niu).}\,\, \,\,\,\,Bei Niu\\
{\small School of Mathematics and Statistics, Xidian University, Xi'an, Shaanxi, 710071, China}\\
}

\maketitle

\begin{abstract}\baselineskip 0.60cm
An equitable tree-$k$-coloring of a graph is a vertex coloring on $k$ colors so that every color class incudes a forest and the sizes of any two color classes differ by at most one.
This kind of coloring was first introduced in 2013 and can be
used to formulate the structure decomposition problem on the communication network with some security considerations.
In 2015, Esperet, Lemoine and Maffray showed that every $d$-degenerate graph admits an equitable tree-$k$-coloring for every $k\geq 3^{d-1}$. Motivated by this result, we attempt to
lower their exponential bound to a linear bound. Precisely,
we prove that every $d$-degenerate graph $G$ admits an equitable tree-$k$-coloring for every $k\geq \alpha d$
provided that $|G|\geq \beta \Delta(G)$, where $(\alpha,\beta)\in \{(8,56), (9,26), (10,18), (11,15), (12,13), (13,12),$\\$ (14,11), (15,10), (17,9), (20,8), (27,7), (52,6)\}$.

\noindent \emph{Keywords: equitable coloring; tree coloring; degenerate graph}.
\end{abstract}

\baselineskip 0.60cm

\section{Introduction }

A graph $G$ is \emph{$d$-degenerate} if every subgraph of $G$ contains a vertex of degree at most $d$.
A \emph{tree-$k$-coloring} of a graph $G$ is a function $c: V(G)\rightarrow \{1,2,\cdots,k\}$ so that $c^{-1}(i)$, the \emph{color class} of $c$ with color $i$, induces a forest for each integer $1\leq i\leq k$. If $c$ is a
tree-$k$-coloring so that $\big||c^{-1}(i)|-|c^{-1}(j)|\big|\leq 1$, then we call $c$ \emph{equitable}. Equivalently, an equitable tree-$k$-coloring is a tree-$k$-coloring so that the size of each color class is at most $\lceil |G|/k \rceil$.
The notion of the equitable tree-$k$-coloring was introduced by Wu, Zhang and Li \cite{WZL.2013}.

In 2015, Esperet, Lemoine and Maffray \cite{ELM.2015} proved that every planar graph has an equitable tree-$k$-coloring for every $k\geq 4$, answering a conjecture of Wu, Zhang and Li \cite{WZL.2013}. Esperet, Lemoine and Maffray \cite{ELM.2015} also proved that every $d$-degenerate graph admits an equitable tree-$k$-coloring for every $k\geq 3^{d-1}$.

In 2017, Chen et al.\,\cite{CGSWW.2017} showed that every $5$-degenerate graph with maximum degree $\Delta$ has an equitable tree-$k$-coloring for every $k\geq (\Delta+1)/2$, partially solving another conjecture of Wu, Zhang and Li \cite{WZL.2013} which states that every graph with maximum degree $\Delta$ has an equitable tree-$k$-coloring for every $k\geq (\Delta+1)/2$. Very recently, Zhang et al.\,\cite{ZNLL.2019} generalized this result by proving that every $d$-degenerate graph with maximum degree $\Delta$ has an equitable tree-$k$-coloring for every $k\geq (\Delta+1)/2$ if $\Delta\geq 10d$.

Motivated by the above results, our goal of this paper is to find for $d$-degenerate graphs
equitable tree-colorings using $O(d)$ colors. The following theorem, which improves the result of Esperet, Lemoine and Maffray \cite{ELM.2015} mentioned above, and partially improves (when $\Delta$ and $|G|$ is sufficient large) the result of Chen et al.\,\cite{CGSWW.2017}, is the main theorem of this paper.

\begin{thm}\label{main-thm}
  Let $\alpha$ and $\beta$ be integers and let $G$ be a $d$-degenerate graph with maximum degree at most $\Delta$ and order $n\geq \beta \Delta$. If $k\geq \alpha d$, then $G$ is equitably tree-$k$-colorable should $\alpha$ and $\beta$ take values from the following table:
  \begin{center}
  \begin{tabular}{c c|c c|c c}
    \hline
    $\alpha$ & $\beta$ & $\alpha$ & $\beta$ & $\alpha$ & $\beta$ \\\hline
    8 & 56 & 12 & 13 & 17 & 9 \\
    9 & 26 & 13 & 12 & 20 & 8 \\
    10 & 18 & 14 & 11 & 27 & 7 \\
    11 & 15 & 15 & 10 & 52 & 6 \\
    \hline
  \end{tabular}
  \end{center}
\end{thm}

Most of the notions and notations in this paper follow from \cite{Bondy.2008}. For example, ${\rm deg}_G(v)$ denotes
the degree of $v$ in a graph $G$, i.e, the number of neighbors of $v$ in $G$, and $e(\mathcal{U},\mathcal{V})$ denotes the number of edges that have one end-vertex in the vertex set $\mathcal{U}$ and the other in the vertex set $\mathcal{V}$.

\section{The proof of Theorem \ref{main-thm}}

In a $d$-degenerate graph $G$, a \emph{degenerate ordering} $v_1,v_2,\cdots,v_n$ of $V(G)$ is a vertex sequence so that $v_i$ has at most $d$ neighbors among $\{v_1,\cdots,v_{i-1}\}$. We use $G_i$ to denote the subgraph of $G$ induced by $\{v_1,\cdots,v_{i}\}$.

Let $t$ be an integer such that $k(t-1)<n\leq kt$. We claim that it is possible to color $v_1,v_2,\cdots,v_n$ of $V(G)$ in a degenerate ordering when $t\leq \beta (2-1/\alpha)$ so that at any stage

$(*)$ every color class induces a forest and contains at most $t$ vertices.

Suppose that we are coloring $v_i$. Let $S_2$ be the set of color classes of $G_{i-1}$ that contains at least two neighbors of $v_i$, and let $S_1$ be the set of other color classes. Clearly, $|S_2|\leq d/2$ and $|S_1|\geq k-d/2$. If there is a color class in $S_1$ containing at most $t-1$ vertices, then move $v_i$ into this color class and we win. Hence we assume that the size of every color class in $S_1$ is exactly $t$. Let $\mathcal{W}_1$ be the set of vertices belonging to some color class of $S_1$. So $|\mathcal{W}_1|=|S_1|t\geq (k-d/2)t$.

Since $|G_{i-1}|<|G|\leq kt$, there is a color class $\mathcal{M}_2\in S_2$ so that $|\mathcal{M}_2|\leq t-1$. If there is a vertex in some color class $\mathcal{M}_1\in S_1$ so that it has at most one neighbor in $\mathcal{M}_2$, then move this vertex into $\mathcal{M}_2$ and then move $v_i$ into $\mathcal{M}_1$. This constructs a coloring of $G_i$ satisfying $(*)$ and we win. Therefore, we are left the case that for any vertex in $\mathcal{W}_1$, it has at least two neighbors in $\mathcal{M}_2$. This implies that
$$\Delta(t-1)\geq |\mathcal{M}_2|\Delta\geq e(\mathcal{M}_2,\mathcal{W}_1)\geq 2|\mathcal{W}_1|\geq (2k-d)t.$$
Since $k\geq \alpha d$ and $n\leq kt$, $(2k-d)t\geq (2-1/\alpha)n$. So
$$t-1\geq \frac{(2k-d)t}{\Delta}\geq \bigg(2-\frac{1}{\alpha}\bigg)\frac{n}{\Delta}\geq \beta\bigg(2-\frac{1}{\alpha}\bigg)$$
by $n\geq \beta \Delta$, a contradiction.

Hence in what follows we always assume that
\begin{align}\label{t}
  t\geq \beta \bigg(2-\frac{1}{\alpha}\bigg)=\frac{(2\alpha-1)\beta}{\alpha}.
\end{align} 
Let
$$t=3^m \omega_1+3^{m-1}\omega_2+\cdots+\omega_{m+1}$$
and
$$\ell_0=0,~~\ell_i=3^{i-1} \omega_1+3^{i-2}\omega_2+\cdots+\omega_{i},~i=1,2,\cdots,m+1$$
where $m$ is an integer and $\omega_1,\omega_2,\cdots,\omega_{m+1}$ are integers chosen from $\{0,1,2\}$ (in particular, $\omega_1\neq 0$).
Actually, those constants $m$ and $\omega_i$'s come from the 3-ary representations of $t$. It is easy to see that
\begin{align}\label{ell}
  \ell_{m+1}=t,~~~~~\ell_{i}=3\ell_{i-1}+\omega_i,~i=1,2,\cdots,m+1
\end{align}

In order to construct an equitable tree-$k$-coloring of $G$, we partition $V(G)$ into $m+1$ disjoint subsets $\mathcal{C}_1,\cdots,\mathcal{C}_m,\mathcal{C}_{m+1}$ so that $\mathcal{C}_i=\mathcal{A}_i\cup \mathcal{B}_i$ for each $1\leq i\leq m$, where $\mathcal{A}_i$ and $\mathcal{B}_i$ are iteratively defined as follows.

Initially we let $\mathcal{A}_0=\mathcal{B}_0=\mathcal{C}_0=\emptyset$.
Suppose that $\mathcal{C}_0,\cdots,\mathcal{C}_{i-1}$ have been constructed. By $H_i$ we denote the subgraph of $G$ induced by $\cup_{j=0}^i \mathcal{C}_j$. For convenience, we write $V(H_i)$, i.e, $\cup_{j=0}^i \mathcal{C}_j$, as $\mathcal{H}_i$.

Arrange the vertices of $G-H_{i-1}$ in a sequence $v^i_1,v^i_2,\cdots,v^i_{n-|H_{i-1}|}$ so that $v^i_j$ has the maximum degree in the graph induced by $V(G)-\mathcal{H}_{i-1}-\{v^i_1,\cdots,v^i_{j-1}\}$. Let
\begin{align}\label{Ai}
 \mathcal{A}_i=\{v^i_1,v^i_2,\cdots,v^i_{(\ell_i-\ell_{i-1})k}\}.
\end{align}

Next, we add vertices into $\mathcal{B}_i$ (starting from $\emptyset$) one by one. Let $\hat{\mathcal{B}}_i$ be the set of vertices that are selected for $\mathcal{B}_i$ so far. If
 there is a vertex in the graph induced by $V(G)-\mathcal{H}_{i-1}-\mathcal{A}_i-\hat{\mathcal{B}}_i$ so that it has at least $(2\alpha-4) d$ neighbors in $\mathcal{H}_{i-1}\cup \mathcal{A}_i\cup \hat{\mathcal{B}}_i$, then put it into $\mathcal{B}_i$, update $\hat{\mathcal{B}}_i$ and repeat this procedure until we cannot find such a vertex.

By the above constructions of $\mathcal{A}_i,\mathcal{B}_i$ and $\mathcal{C}_i$,
\begin{align}
  \label{Hii} |H_i|&=\ell_i k+|\mathcal{B}_1|+\cdots+|\mathcal{B}_i|,\\
  \label{EHi-l}|E(H_i)|&\geq (2\alpha-4)d(|\mathcal{B}_1|+\cdots+|\mathcal{B}_i|).
\end{align}
Since $H_i$ is $d$-degenerate,
\begin{align}\label{EHi-u}
  |E(H_i)|< d|H_i|=d(\ell_i k+|\mathcal{B}_1|+\cdots+|\mathcal{B}_i|).
\end{align}
Hence by \eqref{Hii}, \eqref{EHi-l} and \eqref{EHi-u},
\begin{align}\label{Hi}
|\mathcal{B}_1|+\cdots+|\mathcal{B}_i| <\frac{1}{2\alpha-5}\ell_i k,~~~~~
|H_i|<\frac{2\alpha-4}{2\alpha-5}\ell_i k.
\end{align}

\begin{prop}\label{prop2}
  For $1\leq i\leq m$, if $\Delta_i$ is the maximum degree of the graph $G-H_i$, then
$$\ell_1\Delta+(\ell_2-\ell_1)\Delta_1+(\ell_3-\ell_2)\Delta_2+\cdots+
(\ell_{m+1}-\ell_m)\Delta_m\leq 2\Delta+\frac{10}{3}dt.$$
\end{prop}

\begin{proof}
By the choice of $A_i$ and the definition of $\Delta_i$, we conclude
\begin{align*}
  |E(G)|&\geq \sum_{i=1}^m\sum_{j=1}^{(\ell_i-\ell_{i-1})k} {\rm deg}_{G-G[\mathcal{H}_{i-1}\cup \{v^i_1,\cdots,v^i_{j-1}\}]}v^i_j\\
&\geq \ell_1 k \Delta_1+(\ell_2-\ell_1)k\Delta_2+(\ell_3-\ell_2)k\Delta_3+\cdots+(\ell_m-\ell_{m-1})k\Delta_m.
\end{align*}
Since $G$ is $d$-degenerate, $|E(G)|<dn\leq dkt$. Hence
\begin{align}\label{bas}
  \ell_1  \Delta_1+(\ell_2-\ell_1)\Delta_2+(\ell_3-\ell_2)\Delta_3+\cdots+(\ell_m-\ell_{m-1})\Delta_m<dt.
\end{align}
If $i\geq 3$, then by \eqref{ell} (here note that $\omega_{i+1}\leq 2$, $\omega_i\geq 0$ and $\ell_{i-1}\geq 3$)
\begin{align*}
  \frac{\ell_{i+1}-\ell_i}{\ell_i-\ell_{i-1}}=\frac{3\ell_{i}+\omega_{i+1}-\ell_i}{3\ell_{i-1}+\omega_i-\ell_{i-1}}
\leq \frac{2(3\ell_{i-1}+\omega_i)+2}{2\ell_{i-1}+\omega_i}=3+\frac{2-\omega_i}{2\ell_{i-1}+\omega_i}\leq 3+\frac{1}{\ell_{i-1}}\leq \frac{10}{3}.
\end{align*}
Hence by \eqref{bas}, we have
\begin{align*}
  \notag \frac{10}{3}dt &> \frac{10}{3}\ell_1  \Delta_1+\frac{10}{3}(\ell_2-\ell_1)\Delta_2+
\frac{10}{3}\bigg((\ell_3-\ell_2)\Delta_3+\cdots+(\ell_m-\ell_{m-1})\Delta_m\bigg)\\
\notag&\geq \frac{10}{3}\ell_1\Delta_1+\frac{10}{3}(\ell_2-\ell_1)\Delta_2+(\ell_4-\ell_3)\Delta_3+\cdots+(\ell_{m+1}-\ell_{m})\Delta_m\\
\notag&= \ell_1\Delta+(\ell_2-\ell_1)\Delta_1+(\ell_3-\ell_2)\Delta_2+(\ell_4-\ell_3)\Delta_3+\cdots+(\ell_{m+1}-\ell_m)\Delta_m\\
&~~~~~~~~~~~~~~~~~~~~~~~~~~~~~
-\bigg(\ell_1\Delta+(\ell_2-\frac{13}{3}\ell_1)\Delta_1+(\ell_3-\frac{13}{3}\ell_2+\frac{10}{3}\ell_1)\Delta_2\bigg)
\end{align*}
Now, it is sufficient to prove that
\begin{align}\label{equse}
 \xi:=\ell_1\Delta+(\ell_2-\frac{13}{3}\ell_1)\Delta_1+(\ell_3-\frac{13}{3}\ell_2+\frac{10}{3}\ell_1)\Delta_2\leq 2\Delta.
\end{align}
Since $\ell_1=\omega_1$, $\ell_2=3\ell_1+\omega_2=3\omega_1+\omega_2$ and $\ell_3=3\ell_2+\omega_3=9\omega_1+3\omega_2+\omega_3$ by \eqref{ell},
$$3\xi=3\omega_1\Delta+(3\omega_2-4\omega_1)\Delta_1+(3\omega_3-4\omega_2-2\omega_1)\Delta_2.$$

Recall that $\Delta\geq \Delta_1\geq \Delta_2$ and $\omega_i\in \{0,1,2\}$.

Suppose first that $3\omega_2-4\omega_1\geq 0$.
If $3\omega_3-4\omega_2-2\omega_1\geq 0$, then
$3\xi\leq 3\omega_1\Delta+(3\omega_2-4\omega_1)\Delta+(3\omega_3-4\omega_2-2\omega_1)\Delta=(3\omega_3-3\omega_1-\omega_2)\Delta\leq 6\Delta$. If $3\omega_3-4\omega_2-2\omega_1< 0$, then $3\xi\leq 3\omega_1\Delta+(3\omega_2-4\omega_1)\Delta=(3\omega_2-\omega_1)\Delta\leq 6\Delta$.

Suppose, on the other hand, that $3\omega_2-4\omega_1<0$. If $3\omega_3-4\omega_2-2\omega_1\leq 0$, then $3\xi\leq 3\omega_1\Delta\leq 6\Delta$. If $3\omega_3-4\omega_2-2\omega_1>0$, then
$3\xi\leq 3\omega_1\Delta+(3\omega_2-4\omega_1)\Delta_1+(3\omega_3-4\omega_2-2\omega_1)\Delta_1
=3\omega_1\Delta+(3\omega_3-\omega_2-6\omega_1)\Delta_1$. If $3\omega_3-\omega_2-6\omega_1\leq 0$, we then have
$3\xi\leq 3\omega_1\Delta\leq 6\Delta$. If  $3\omega_3-\omega_2-6\omega_1\geq 0$, then
$3\xi\leq 3\omega_1\Delta+(3\omega_3-\omega_2-6\omega_1)\Delta=(3\omega_3-3\omega_1-\omega_2)\Delta\leq 6\Delta$.

Therefore, in each case we conclude that $3\xi\leq 6\Delta$, and \eqref{equse} holds.
\end{proof}

In what follows, we color the vertices of $\mathcal{C}_1,\mathcal{C}_2,\cdots,\mathcal{C}_m,\mathcal{C}_{m+1}$ in such a sequence.

\begin{clm}\label{clm1}
  We can tree-$k$-color the vertices of $\mathcal{C}_1$ in a degenerate ordering so that each color class contains at most
$\lceil \frac{2\alpha-3}{2\alpha-5}\ell_1\rceil$ vertices.
\end{clm}

\begin{proof}

Since
\begin{align*}
  \frac{k-\lfloor\frac{d}{2}\rfloor}{k}\frac{2\alpha}{2\alpha-1}\geq \frac{2k-d}{k}\frac{\alpha}{2\alpha-1}
  \geq \frac{2\alpha d-d}{\alpha d}\frac{\alpha}{2\alpha-1}=1,
\end{align*}
by \eqref{Hi}, we have
\begin{align*}
  |\mathcal{C}_1|=|H_1|<\frac{2\alpha-4}{2\alpha-5}\ell_1 k\leq \frac{2\alpha-4}{2\alpha-5}\ell_1 k \frac{k-\lfloor\frac{d}{2}\rfloor}{k}\frac{2\alpha}{2\alpha-1}
  &=\frac{2\alpha-4}{2\alpha-5}\frac{2\alpha}{2\alpha-1}\ell_1 \big(k-\lfloor\frac{d}{2}\rfloor\big)\\
  &<\frac{2\alpha-3}{2\alpha-5}\ell_1 \big(k-\lfloor\frac{d}{2}\rfloor\big).
\end{align*}
So, when each vertex $u\in \mathcal{C}_1$ is being colored,
there are at least $k-\lfloor d/2\rfloor$ color classes containing at most one neighbor of $u$, at least one of which contains less than $\frac{2\alpha-3}{2\alpha-5}\ell_1$ vertices. Hence we can put $u$ into such a color class and we win.
\end{proof}

%


\begin{clm}\label{clm2}
  Let $2\leq i\leq m+1$ be an integer and let $$L_i=\left\{
       \begin{array}{ll}
         \lceil \frac{2\alpha-3}{2\alpha-5}\ell_i\rceil , & \hbox{$1\leq i\leq m$;} \\
         t, & \hbox{$i=m+1$.}
       \end{array}
     \right.
$$
If $H_{i-1}$ has been tree-$k$-colored so that every color class of $H_{i-1}$ contains at most $L_{i-1}$
vertices, then it is possible to tree-$k$-color the vertices of $\mathcal{C}_i$ in a degenerate ordering so that

(1) every color class of $H_i$ contains at most $L_i$
vertices;

(2) no vertex in $H_{i-1}$ would be recolored.
  \end{clm}

If Claim \ref{clm2} has been proved, then by Claims \ref{clm1} and \ref{clm2}, one can conclude that
all the vertices of $\mathcal{C}_1\cup \mathcal{C}_2\cup \cdots \cup \mathcal{C}_{m+1}=V(G)$ can be tree-$k$-colored so that every color class of $H_{m+1}=G$ contains at most $L_{m+1}=t$ vertices. This just gives an equitable tree-$k$-coloring of $G$ and we complete the proof.

Therefore, the remaining task is to confirm Claim \ref{clm2}.
Before proving it, we first show the following proposition.

\begin{prop}\label{prop1}
  For $2\leq i\leq m+1$, $$\frac{L_{i-1}}{L_i}\leq \frac{1}{2}.$$
\end{prop}

\begin{proof}
Here we recall \eqref{ell} which states that $\ell_i=3\ell_{i-1}+\omega_i\geq 3\ell_{i-1}$ and $\ell_{m+1}=t$.

If $2\leq i\leq m$, then we consider two subcases.
If $\ell_{i-1}\geq 2$, then $\ell_i\geq 6$ and thus
\begin{align*}
\frac{L_{i-1}}{L_i}=\frac{\lceil \frac{2\alpha-3}{2\alpha-5}\ell_{i-1}\rceil}{\lceil \frac{2\alpha-3}{2\alpha-5}\ell_i\rceil}
&\leq \frac{\frac{2\alpha-3}{2\alpha-5}\ell_{i-1}
+\frac{2\alpha-6}{2\alpha-5}}{\frac{2\alpha-3}{2\alpha-5}\ell_i}\\
&\leq \frac{1}{3}+\frac{2\alpha-6}{6(2\alpha-3)}\\
&=\frac{3\alpha-6}{6\alpha-9}
<\frac{1}{2}.
\end{align*}
If $\ell_{i-1}=1$, then $\ell_{i}\geq 3$ and
\begin{align*}
\frac{L_{i-1}}{L_i}=\frac{\lceil \frac{2\alpha-3}{2\alpha-5}\ell_{i-1}\rceil}{\lceil \frac{2\alpha-3}{2\alpha-5}\ell_i\rceil}
\leq \frac{\lceil \frac{2\alpha-3}{2\alpha-5}\rceil}{\lceil 3\cdot \frac{2\alpha-3}{2\alpha-5}\rceil}\leq \frac{2}{4}=\frac{1}{2}.
\end{align*}

If $i=m+1$, then $\ell_i=t$, $\ell_m/t\leq 1/3$ and
\begin{align*}
\frac{L_{i-1}}{L_i}=\frac{\lceil \frac{2\alpha-3}{2\alpha-5}\ell_{i-1}\rceil}{t}
&\leq \frac{\frac{2\alpha-3}{2\alpha-5}\ell_{m}
+\frac{2\alpha-6}{2\alpha-5}}{t}\\
&\leq \frac{2\alpha-3}{3(2\alpha-5)}
+\frac{\alpha(2\alpha-6)}{\beta(2\alpha-5)(2\alpha-1)}\\
&=\frac{10\alpha^2-22\alpha+6}{24\alpha^2-72\alpha+30}<\frac{1}{2}
\end{align*}
Recall that $t\geq \beta (2\alpha-1)/\alpha$ by \eqref{t} and $\beta\geq 6$. 
\end{proof}

\begin{proof}[Proof of Claim \ref{clm2}]

We now attempt to tree-$k$-color the vertices of $\mathcal{C}_i$ in a degenerate ordering so that (1) and (2) hold. Suppose that we are coloring a vertex $u\in \mathcal{C}_i$ and let $c$ be the current partial coloring of $G$ we have already obtained so that every color class of $c$ contains at most $L_i$
vertices.

Define an auxiliary digraph $\mathfrak{D}:=\mathfrak{D}(c)$ on the color classes
of $c$ by $\mathcal{X}\mathcal{Y}\in E(\mathfrak{D})$ if and only if some vertex $x\in \mathcal{X}-\mathcal{H}_{i-1}$
has at most one neighbor in $\mathcal{Y}$. In this case we say that $x$
\emph{witnesses} $\mathcal{XY}$.
If $P:=\mathcal{M}_l\mathcal{M}_{l-1}\cdots \mathcal{M}_1\mathcal{M}_0$ is a path in $\mathfrak{D}$ and $x_i$ ($l\geq i\geq 1$) is a vertex in $\mathcal{M}_i$ such that $x_i$ witnesses $\mathcal{M}_i\mathcal{M}_{i-1}$, then \emph{switching witnesses} along $P$ means moving $x_i$ to $\mathcal{M}_{i-1}$ for every $l\geq i\geq 1$. This operation decreases
$|\mathcal{M}_l|$ by one and increases $|\mathcal{M}_0|$ by one, while leaving $|\mathcal{M}_i|$ with $2\leq i\leq l-1$ unchanged.

Let $Y_0$ be the set of color classes of $c$ containing less than $L_i$ vertices. By $Y_i$ $(i\geq 1)$, we denote the set of color classes of $c$ such that

(i) $Y_i\cap \bigcup_{j=0}^{i-1}Y_j=\emptyset$, and

(ii) for any color class $\mathcal{M}_i\in Y_i$ there exists a color class $\mathcal{M}_{i-1}\in Y_{i-1}$ so that $\mathcal{M}_{i}\mathcal{M}_{i-1}\in E(\mathfrak{D})$.

Let $\mathfrak{Y}=\bigcup Y_j$ and let $y=|\mathfrak{Y}|$. We claim that there exists one color class $\mathcal{M}_j\in Y_j\in \mathfrak{Y}$ containing at most one neighbor of $u$. Hence by moving $u$ into $\mathcal{M}_j$ and switching witnesses along $P=\mathcal{M}_j\mathcal{M}_{j-1}\cdots \mathcal{M}_0$, where $\mathcal{M}_t\in Y_t$ ($j\geq t\geq 0$), we can complete the coloring on $u$ so that (1) and (2) hold.

Suppose, to the contrary,  that every color class of $\mathfrak{Y}$ contains at least two neighbors of $u$. Note that any vertex $v\in \mathcal{C}_i$ in some color class outside of $\mathfrak{Y}$ has at least two neighbors in every color class of $\mathfrak{Y}$, because otherwise the color class containing $v$ will in included in $\mathfrak{Y}$. In the following, we try our luck to find contradictions under those assumptions.

We claim that $y$ is upper-bounded. Actually, there are less than $(2\alpha-4)d$ neighbors of $u$ in $H_{i-1}$ (otherwise $u$ would have already been selected for $\mathcal{B}_{i-1}$ and thus $u\in \mathcal{C}_{i-1}$), and in $\mathcal{C}_i$ there are at most $d$ neighbors of $u$ that are already colored (recall that the vertices of $\mathcal{C}_i$ are being colored in a degenerate ordering). Therefore, among the neighbors of $u$, less than $(2\alpha-3)d$ are colored under $c$. This implies that there are less than $(2\alpha-3)d/2$ color classes that contain at least two neighbors of $u$. Hence
\begin{align}\label{y-bonud-1}
  y<\frac{2\alpha-3}{2}d, ~~~~~~\frac{y}{d}<\frac{2\alpha-3}{2}.
\end{align}

Let $\mathcal{S}$ be the set of vertices that are contained in some color class of $\mathfrak{Y}$, and let $\mathcal{T}$
be the set of colored vertices in $\mathcal{C}_i$ that do not belong to $\mathcal{S}$.
By the $d$-degeneracy of $G$ and by the above analysis, we have
\begin{align*}
  d(|\mathcal{S}|+|\mathcal{T}|)> e(\mathcal{T},\mathcal{S})\geq 2y|\mathcal{T}|,
\end{align*}
which implies
\begin{align}\label{st}
  (2y-d)|\mathcal{T}|< d|\mathcal{S}|.
\end{align}

By the definition of $Y_0$, every color class of $Y_0$ contains less than $L_i$ vertices
and every other color class in $c$ contains exactly $L_i$ vertices.
Since no vertex in $H_{i-1}$ would be recolored when coloring $\mathcal{C}_i$,  every color class in $c$ has at most $L_{i-1}$ vertices in $H_{i-1}$. So
\begin{align*}
 |\mathcal{S}|\leq y L_i,~~|\mathcal{T}|\geq (k-y)L_i-(k-y)L_{i-1}=(k-y)(L_i-L_{i-1})\geq (\alpha d-y)(L_i-L_{i-1}),
\end{align*}
which implies by \eqref{st} that
\begin{align*}
  (2y-d)(\alpha d-y)(L_i-L_{i-1})<d y L_i.
\end{align*}
Write $\gamma=y/d$. We deduce from Proposition \ref{prop1} and the above inequality that
\begin{align*}
  f(\gamma):=2\gamma^2-(2\alpha-1)\gamma+\alpha>0.
\end{align*}
Since
\begin{align*}
  f\bigg(\frac{2\alpha-3}{2}\bigg)=3-\alpha<0,~~~~~
f\bigg(\frac{\alpha}{2\alpha-3}\bigg)=\frac{2\alpha(3-\alpha)}{(2\alpha-3)^2}<0,
\end{align*}
we conclude by \eqref{y-bonud-1} that
\begin{align}\label{y-bound-2}
  \frac{y}{d}=\gamma<\frac{\alpha}{2\alpha-3}.
\end{align}

We now count the number $\zeta_c$ of vertices that have already been colored under $c$. Actually, among the $k$ color classes, there are only $|Y_0|\leq y$ color classes containing less than $L_i$ vertices. Therefore, $\zeta_c\geq (k-y)L_i$.
Since $k\geq \alpha d$ and $y<\alpha d/(2\alpha-3)$ by \eqref{y-bound-2},
\begin{align*}
  \zeta_c\geq \frac{2\alpha-4}{2\alpha-3}k L_i= \frac{2\alpha-4}{2\alpha-3} k \bigg\lceil \frac{2\alpha-3}{2\alpha-5}\ell_i\bigg\rceil\geq \frac{2\alpha-4}{2\alpha-5}\ell_i k
\end{align*}
if $2\leq i\leq m$. On the other hand, it is trivial that $$\zeta_c\leq |H_i|<\frac{2\alpha-4}{2\alpha-5}\ell_i k$$ by \eqref{Hi}. This is a contradiction.

Hence there remains only one case: $i=m+1$.

Recall that we are now coloring a vertex $u\in \mathcal{C}_{m+1}$ and $c$ is the partial coloring of $G$ already constructed  with the property that every color class of $c$ contains at most $L_{m+1}=t$ vertices.
Since $|V(G)-\{u\}|<n\leq kt$, there is at least one color class in $c$ that contains less than $t$ vertices. This implies that $Y_0\neq \emptyset$.

Let $\mathcal{M}$ be a color class of $Y_0$. For $1\leq j\leq m$, let $\mathcal{Z}_j=\mathcal{M}\cap \mathcal{C}_j$ and $z_j=|\mathcal{Z}_j|$.
Since no vertex in $H_{j-1}$ would be recolored while coloring vertices of $\mathcal{C}_{j}$ and every color class of $H_{j-1}$ contains at most $L_{j-1}$ vertices,
\begin{align}\label{Z}
  \sum_{s=1}^j z_s\leq L_j, ~~ 1\leq j\leq m.
\end{align}.

Let $\mathcal{U}$ be the set of colored vertices in $\mathcal{C}_{m+1}$
that are adjacent to some vertex in $\mathcal{M}$. 

For $1\leq j\leq m$, recall that $\Delta_j$ is the maximum degree of the graph $G-H_j$.
It is easy to see that
\begin{align}\label{u}
  |\mathcal{U}|\leq z_1\Delta+z_2\Delta_1+\cdots+z_{m+1}\Delta_m
\end{align}
Since $kt/\beta\geq n/\beta\geq \Delta\geq \Delta_1\geq\cdots\geq \Delta_m$, by \eqref{Z}, \eqref{u} and Propostion \ref{prop2}, we have
\begin{align}
  \notag |\mathcal{U}|&\leq L_1\Delta+(L_2-L_1)\Delta_1+\cdots+(L_{m+1}-L_m)\Delta_m\\
\notag &= \bigg\lceil \frac{2\alpha-3}{2\alpha-5}\ell_1\bigg\rceil \Delta
+\bigg(\bigg\lceil \frac{2\alpha-3}{2\alpha-5}\ell_2\bigg\rceil-\bigg\lceil \frac{2\alpha-3}{2\alpha-5}\ell_1\bigg\rceil\bigg)\Delta_1
+\cdots+\bigg(t-\bigg\lceil \frac{2\alpha-3}{2\alpha-5}\ell_m\bigg\rceil\bigg)\Delta_m\\
\notag &=\bigg\lceil \frac{2\alpha-3}{2\alpha-5}\ell_1\bigg\rceil (\Delta-\Delta_1)
+\bigg\lceil \frac{2\alpha-3}{2\alpha-5}\ell_2\bigg\rceil (\Delta_1-\Delta_2)+\cdots
+\bigg\lceil \frac{2\alpha-3}{2\alpha-5}\ell_m\bigg\rceil (\Delta_{m-1}-\Delta_m)+t\Delta_m\\
\notag &\leq \bigg(\frac{2\alpha-3}{2\alpha-5}\ell_1+\frac{2\alpha-6}{2\alpha-5}\bigg) (\Delta-\Delta_1)
+\bigg(\frac{2\alpha-3}{2\alpha-5}\ell_2+\frac{2\alpha-6}{2\alpha-5}\bigg) (\Delta_1-\Delta_2)\\
\notag &~~~~~~~~~~~~~~~~~~~~~~~~~~~~~~~~~~~~~~+\cdots
+\bigg(\frac{2\alpha-3}{2\alpha-5}\ell_m+\frac{2\alpha-6}{2\alpha-5}\bigg) (\Delta_{m-1}-\Delta_m)+t\Delta_m\\
\notag&<  \frac{2\alpha-6}{2\alpha-5}\Delta
+\frac{2\alpha-3}{2\alpha-5}\bigg(\ell_1\Delta+(\ell_2-\ell_1)\Delta_1+(\ell_3-\ell_2)\Delta_2+\cdots+
(\ell_{m+1}-\ell_m)\Delta_m\bigg)\\
 \notag&\leq \frac{2\alpha-6}{2\alpha-5}\Delta+\frac{2\alpha-3}{2\alpha-5}\bigg(2\Delta+\frac{10}{3}dt\bigg)
=\frac{6\alpha-12}{2\alpha-5}\Delta+\frac{20\alpha-30}{6\alpha-15}dt\\
\label{upper}&\leq \bigg(\frac{6\alpha-12}{(2\alpha-5)\beta}k+\frac{20\alpha-30}{6\alpha-15}d\bigg)t.
\end{align}

On the other hand, recall that any vertex $v\in \mathcal{C}_{m+1}$ in some color class outside of $\mathfrak{Y}$ has at least two neighbors in every color class (e.g.,\,$\mathcal{M}$) of $\mathfrak{Y}$. Therefore, all vertices lying in the $(k-y)$ color classes outside of $\mathfrak{Y}$ are neighbors of $\mathcal{M}$. By the definition of $Y_0$, we also know that each of those $(k-y)$ color classes in $c$ contains exactly $L_{m+1}$ vertices. Since no vertex in $H_{m}$ would be recolored when coloring $\mathcal{C}_{m+1}$,  every color class in $c$ has at most $L_{m}$ vertices in $H_{m}$. Therefore,
\begin{align}
  \notag |\mathcal{U}|&\geq (k-y)(L_{m+1}-L_m)=(k-y)\bigg(t-\bigg\lceil \frac{2\alpha-3}{2\alpha-5}\ell_m\bigg\rceil\bigg)\\
\notag &\geq (k-y)\bigg(1-\frac{2\alpha-3}{2\alpha-5} \frac{\ell_m}{t} - \frac{2\alpha-6}{2\alpha-5} \frac{1}{t}  \bigg) t\\
\notag &\geq (k-y)\bigg(1-\frac{2\alpha-3}{6\alpha-15} - \frac{2\alpha-6}{2\alpha-5} \frac{\alpha}{(2\alpha-1)\beta}  \bigg) t\\
\label{lower} &\geq \bigg(k-\frac{\alpha}{2\alpha-3}d \bigg)\bigg(\frac{4\alpha-12}{6\alpha-15}-\frac{\alpha(2\alpha-6)}{(2\alpha-1)(2\alpha-5)\beta}\bigg)t
\end{align}
Here we use \eqref{y-bound-2} along with the facts that $\ell_m/t=\ell_m/\ell_{m_1}\leq 1/3$ and $t\geq (2\alpha-1)\beta/\alpha$.

Combining \eqref{lower} with \eqref{upper}, we immediately conclude
\begin{align*}
\bigg(k-\frac{\alpha}{2\alpha-3}d \bigg)\bigg(\frac{4\alpha-12}{6\alpha-15}-\frac{\alpha(2\alpha-6)}{(2\alpha-1)(2\alpha-5)\beta}\bigg)\leq
\frac{6\alpha-12}{(2\alpha-5)\beta}k+\frac{20\alpha-30}{6\alpha-15}d,
\end{align*}
which implies
\begin{align}
 \label{kd1} \frac{k}{d}&\leq \frac{\frac{20\alpha-30}{6\alpha-15}+\frac{\alpha(4\alpha-12)}{(2\alpha-3)(6\alpha-15)}-
  \frac{\alpha^2(2\alpha-6)}{(2\alpha-1)(2\alpha-3)(2\alpha-5)\beta}}
{\frac{4\alpha-12}{6\alpha-15}-\frac{\alpha(2\alpha-6)}{(2\alpha-1)(2\alpha-5)\beta}-\frac{6\alpha-12}{(2\alpha-5)\beta}}\\
\label{kd}&=\frac{(44\alpha^3-154\alpha^2+156\alpha-45)\beta-(3\alpha^3-9\alpha^2)}{(8\alpha^3-40\alpha^2+54\alpha-18)\beta-
(42\alpha^3-171\alpha^2+198\alpha-54)}
\end{align}
Note that the numerators and denominators in \eqref{kd1} and \eqref{kd} are positive if $\alpha$ and $\beta$ are chosen from the given table.

Since $k\geq \alpha d$, we deduce from \eqref{kd} that
\begin{align*}
  (8\alpha^4-84\alpha^3+208\alpha^2-174\alpha+45)\beta\leq 42\alpha^4-174\alpha^3+207\alpha^2-54\alpha.
\end{align*}
However, we will get a contradiction if we choose $\alpha$ and $\beta$ from the given table. This ends the proof of Claim \ref{clm2} and thus the whole proof of Theorem \ref{main-thm} completes.
\end{proof}

\section*{Acknowledgements}

We would like to thank the paper \cite{KNP.2005} contributed by Kostochka, Nakprasit and Pemmaraju, because the idea of the proof of Theorem \ref{main-thm} mainly comes from there. Perhaps, using our similar arguments that involve parameters $\alpha$ and $\beta$, one is possible to improve or generalize their early result in  \cite{KNP.2005}.

\bibliographystyle{srtnumbered}
\bibliography{mybib}

\end{document}